\documentclass[a4paper, 11pt]{amsart}

\usepackage[T1]{fontenc}
\usepackage[utf8]{inputenc}

\usepackage{amssymb}

\makeatletter
\@namedef{subjclassname@2010}{%
  \textup{2010} Mathematics Subject Classification}
\makeatother

\newcommand{\ov}[1]{\overline{#1}}

\allowdisplaybreaks[2]

\newtheorem{theorem}{Theorem}[section]

\newtheorem{corollary}[theorem]{Corollary}
\theoremstyle{definition}

\newtheorem{definition}[theorem]{Definition}
\newtheorem{example}[theorem]{Example}

\numberwithin{equation}{section}

\begin{document}

\author{Matej Brešar}
\address{ Faculty of Mathematics and Physics,  University of Ljubljana,  and Faculty of Natural Sciences and Mathematics, University
of Maribor, Slovenia}
\email{matej.bresar@fmf.uni-lj.si}


\thanks{The author was supported by ARRS Grant P1-0288.}

\title[Functional identities and zLpd Banach algebras]{Functional identities and zero Lie product determined  Banach algebras}

\begin{abstract} Three problems connecting functional identities to the recently introduced notion of a zero Lie product determined Banach algebra are discussed. The first one concerns commuting linear maps, the second one concerns derivations that preserve commutativity, and the third one concerns bijective commutativity preserving linear maps.
\end{abstract}
\maketitle

\section{Introduction}

A Banach algebra $A$ is said to be {\em zero Lie product determined} ({\em zLpd} for short)
if  every continuous bilinear functional $\varphi: A \times A\to \mathbb C$ with the property that for all $x,y\in A$,
\begin{equation}\label{0e}[x,y]=0\implies \varphi(x,y)=0,\end{equation}
is of the form \begin{equation}\label{1e}\varphi(x, y) = \omega([x,y])\end{equation} for
some continuous linear functional $\omega$ on $A$  (here, $[x,y]$ stands for $xy-yx$). This notion was introduced in \cite{ABEV0}. The main message of this and the subsequent paper \cite{ABEV} is that the class of  zLpd Banach algebras is quite large; in particular, it includes $C^*$-algebras
and group algebras $L ^1(G)$ where $G$ is a locally compact group. We also mention a little technical result, which will be tacitly used throughout the paper: a Banach algebra
$A$ is zLpd if and only if  for any Banach space $X$, every continuous bilinear map $\varphi: A\times A \to 
X$ satisfying
\eqref{0e} is of the form \eqref{1e} for some 
continuous linear map $\omega : [A, A] \to X$ \cite[Proposition 3.1]{ABEV0}.

Replacing the Lie product $[x,y]$ by the ordinary product $xy$ in the above definition, one obtains the definition of a 
 zero product determined (zpd) Banach algebra.  These Banach algebras have been studied extensively
over the last decade (as a matter of fact, most papers deal with a slightly broader class of Banach algebras
 with property $\mathbb B$). Besides providing examples of such Banach algebras, a lot of work has been done in finding applications, i.e., in showing that various problems can be successfully
resolved in zpd Banach algebras (see the recent papers \cite{ABESV2, AEV} and references therein). On the other hand, the  two papers on zLpd Banach algebras \cite{ABEV0, ABEV} do not provide applications, so it may not be clear thus far whether this notion is indeed useful. Our goal is to remedy this situation. More specifically, we will consider three problems related to functional identities, and show that they can be solved in zLpd Banach algebras that satisfy some  technical conditions. We begin, in Section \ref{s2}, by providing the necessary information on functional identities, and then discuss these three problems in Sections \ref{s3}-\ref{s5}.

In Section \ref{s3}, we consider commuting  maps, i.e.,  maps $f$ from an algebra $A$ to itself that satisfy $[f(x),x]=0$ for all $x\in A$. These maps have played a crucial role in the development of  the theory of functional identities (see \cite{FIbook}).  Under the assumption $A$ is a semiprimitive (= semisimple)
zLpd Banach algebra, we will show that  every commuting continuous linear map $f$ can be written as a sum $f= \gamma  + \mu$ where $\gamma$ lies in the 
centroid of the Lie algebra $A^-$ (i.e., $\gamma([x,y])= [x,\gamma(y)] =[\gamma(x),y]$ for all $x,y\in A$) and $\mu$ maps into the center of $A$.

There are many results in the literature, both in the algebraic and the analytic context, that consider  commutativity conditions on the range of derivations. One of the earliest ones, due to Herstein,
states that the range of a nonzero derivation $d$ of a prime ring $R$ is commutative (i.e., $[d(x),d(y)]=0$ for all $x,y\in R$)  only if the characteristic of $R$ is $2$ and $R$ can be embedded into $M_2(F)$, the ring of $2\times 2$ matrices over a field $F$ \cite{Her2}. This is a fairly easy result, which, however, stimulated further research by many authors (see, e.g., \cite{L, PL}). In Section \ref{s4}, we will discuss nonzero  derivations $d$ on  a Banach algebra $A$ that satisfy $[d(x),d(y)]=0$ only when $x$ and $y$ commute (that is,   derivations that preserve  commutativity). It is easy to see that all derivations of the algebra $M_2(\mathbb C$) satisfy this condition, and we will show that, under appropriate conditions, a zLpd Banach algebra $A$ having such a derivation $d$
must be close (or even isomorphic) to $M_2(\mathbb C)$.

Our deepest result is established in Section \ref{s5} where we consider bijective commutativity preserving linear maps between Banach algebras. The study of such maps has a long history. For details, we refer to the survey paper \cite{Sem} 
(in particular, see the discussion around \cite[Theorem 2.2]{Sem} which is
 one of the crucial results that  inspired the study of the zpd and zLpd   properties). It is also noteworthy that the interest in commutativity preservers on operator algebras has recently reemerged \cite{Ham, Mol}, and, as the author  learned from a private communication with J.  Hamhalter, the characterization of bijective linear maps between von Neumann algebras that preserve commutativity in both directions  from \cite{BM} should be useful in these new studies. In the context of zLpd Banach algebras satisfying certain natural constraints (which have a clear description in von Neumann algebras), we will  consider maps that preserve commutativity in one direction only.  Our result states that either the map under consideration preserves commutativity
in both directions, in which case it has a standard form expressible   in terms of (anti)isomorphisms, or it maps a noncentral ideal into the center. 
An example of the latter possibility is provided.


\section{Functional identities preliminaries} \label{s2}

Let $Q$ be a unital ring with center $Z(Q)$ and let $R$
be a nonempty subset of $Q$.  
Let $m$ be a positive integer. 
For elements $x_i \in R$, $i=1,2,\ldots,m$, 
we write
\begin{eqnarray*}
\ov{x}_m &=& (x_1,\ldots,x_m)\in{R}^m,\\
\ov{x}_{m}^{i}&=& (x_1,\ldots,x_{i-1},x_{i+1},\ldots,x_m)\in{R}^{m-1},\\
\ov{x}_{m}^{ij}= \ov{x}_{m}^{ji} &=& (x_1,\ldots,x_{i-1},x_{i+1},\ldots,
x_{j-1},x_{j+1}\ldots,x_m)\in{R}^{m-2}.
\end{eqnarray*}
Let $I,J$ be subsets of $\{1,2,\ldots,m\}$, and for each $i\in I$ and $j\in J$
let  
$$E_i:R^{m-1}\to Q\quad\mbox{and}\quad F_j:R^{m-1}\to Q$$
be arbitrary maps (if $m=1$, we  consider $E_i$ and 
$F_j$ as elements of $Q$). 
The basic functional identities, on which the general theory is based, are
\begin{eqnarray}
\sum_{i\in I}E_i(\ov{x}_m^i)x_i+ \sum_{j\in J}x_jF_j(\ov{x}_m^j)
& = &
0\quad\mbox{for all $\ov{x}_m\in R^m$},\label{3S1}\\
\sum_{i\in I}E_i(\ov{x}_m^i)x_i+ \sum_{j\in J}x_jF_j(\ov{x}_m^j)
& \in & Z(Q)\quad\mbox{for all $\ov{x}_m\in R^m$}.\label{3S2}
\end{eqnarray}
Of course, (\ref{3S1}) implies (\ref{3S2}). Thus, one should not understand
that (\ref{3S1}) and (\ref{3S2}) are satisfied simultaneously by the same maps $E_i$
and $F_j$. Each of the two identities has to be treated separately.  

We define the  {\em standard solution} of functional identities  (\ref{3S1}) and (\ref{3S2})  as
\begin{eqnarray}\label{3S3}
E_i(\ov{x}_m^i) & = & \sum_{j\in J,\atop
j\not=i}x_jp_{ij}(\ov{x}_m^{ij})
+\lambda_i(\ov{x}_m^i),\quad i\in I,\nonumber\\
F_j(\ov{x}_m^j) & = & -\sum_{i\in I,\atop
i\not=j}p_{ij}(\ov{x}_m^{ij})x_i
-\lambda_j(\ov{x}_m^j),\quad j\in J,\\
&   & \lambda_k=0\quad\mbox{if}\quad k\not\in I\cap J,\nonumber
\end{eqnarray}
where
\begin{eqnarray*}
&  & p_{ij}:R^{m-2}\to Q,\;\;
i\in I,\; j\in J,\; i\not=j,\nonumber\\
&  & \lambda_k:R^{m-1}\to Z(Q),\;\; k\in I\cup J
\end{eqnarray*}
are any maps
 (if $m=1$, this should  be understood as   that $p_{ij} =0$ and 
$\lambda_k$ is an element in $Z(Q)$). Note that \eqref{3S3} indeed implies  (\ref{3S1}), and hence also  (\ref{3S2}).

The case where one of the sets $I$ and $J$ is empty is not excluded.
By convention, the sum over
$\emptyset$ is $0$. Thus, if  
 $J=\emptyset$,
(\ref{3S1})  reads as
$$
\sum_{i\in I}E_i(\ov{x}_m^i)x_i = 0\quad\mbox{for all $\ov{x}_m\in R^m$,}
$$
and the standard solution of this functional identity is $E_i=0$ for all $i\in I$. Similarly, the standard solution of
$$\sum_{j\in J}x_jF_j(\ov{x}_m^j) = 
0\quad\mbox{for all $\ov{x}_m\in R^m$}
$$
 is $F_j=0$ for each $j$.

We are interested in  sets 
$R$ with the property that 
the functional identities (\ref{3S1}) and (\ref{3S2}) have only standard solutions, provided that 
the index sets
$I$ and $J$ are small enough. The following definition describes this precisely.

\begin{definition}\label{def1}Let $d$ be a positive integer. We  say that $R$ is a  {\em
$d$-free subset of $Q$} if  
the following two conditions hold for all $m\ge 1$ and all
 $I,J\subseteq \{1,2,\ldots,m\}$:
\begin{enumerate}
\item[(a)]If
$\max\{|I|,|J|\}\le d$, then (\ref{3S1}) implies (\ref{3S3}).
\item[(b)] If
$\max\{|I|,|J|\}\le d-1$, then (\ref{3S2}) implies (\ref{3S3}).
\end{enumerate}
If $R=Q$ is a $d$-free subset of itself, then we simply say that $R$ is {\em $d$-free}. 
\end{definition}
Note that $d$-freeness implies $d'$-freeness for every $d'< d$.

One of the fundamental theorems on functional identities states that a prime ring $R$ is a $d$-free (with $d\ge 2$) subset of $Q_{ml}(R)$,  the  maximal left  ring of quotients of $R$, if and only if $R$ cannot be embedded into the ring of $(d-1)\times (d-1)$ matrices over a field
  \cite[Theorem 5.11]{FIbook} (equivalently, $R$ does not satisfy a polynomial identity of degree less than $2d$). Another important result states that if $S$ is any unital ring, then  $R=M_d(S)$, the ring of $d\times d$ matrices over $S$,  is $d$-free \cite[Corollary 2.22]{FIbook}.

In this paper, however, we are  interested in  $d$-freeness of Banach algebras. The $C^*$-algebra case is relatively well-understood. For example, a prime $C^*$-algebra $A$ is a $d$-free subset of $Q_{ml}(A)$ if and only if $A$ is not isomorphic to the matrix algebra  $M_k(\mathbb C)$ with $k< d$. This follows from the  theorem stated in the preceding paragraph, 
Posner's theorem on prime rings satisfying a polynomial identity \cite{P2}, and the standard description of finite-dimensional $C^*$-algebras. Further, a von Neumann algebra is $d$-free provided that it has no central summands of 
type $I_1,\dots,I_{d-1}$ \cite[Theorem 3.2 and Lemma 3.12]{ABV}.

At present, not much is known about $d$-freeness of  other zLpd Banach algebras, like group algebras $L ^1(G)$. 
Our main results will therefore have clear applications  only for $C^*$-algebras, although they are
potentially applicable to a much broader class of Banach algebras.


\section{Commuting linear maps}\label{s3}

The { \em centroid} of a Lie algebra $L$, denoted Cent$(L)$, is the space of all linear maps $\gamma:L\to L$ satisfying 
\begin{equation*} \gamma([x,y])= [x,\gamma(y)] =[\gamma(x),y]\quad\mbox{for all $x,y\in L$.}\end{equation*}
This is a classical notion in the theory of Lie and other nonassociative algebras (see, e.g., 
\cite[Section X.1]{J} and
\cite[Sections 1.6 and 1.7]{Mc}).
We are only interested in the case where $A$ is an associative algebra and $L=A^-$, i.e., $L$ is the vector space $A$ endowed with the Lie product $[x,y]=xy-yx$.

Recall that a map $f:A\to A$ is said to be commuting if $[f(x),x]=0$ for all $x\in A$. 
It is obvious that every map in Cent$(A^-)$ is commuting, and so is every map having the range in $Z(A)$, the center of $A$. Also, the sum of two commuting maps is  commuting.
The ``if'' part of the following theorem is thus trivial. 

\begin{theorem}\label{tco}
Let $A$ be a semiprimitive zLpd Banach algebra. Then a continuous linear map $f:A\to A$ is commuting if and only if there exist $\gamma\in {\rm Cent}(A^-)$ and  a linear map $\mu:A\to Z(A)$ such that $f=\gamma + \mu$.
\end{theorem}

\begin{proof} We begin the proof of the nontrivial ``only if'' part by invoking \cite[Lemma 3.1]{BZ} which tells us that a commuting linear map $f$ on any algebra $A$ satisfies 
\begin{equation}\label{e}\big[[w,z], [u, f([x,y]) - [x, f(y)]]\big]=0 \quad\mbox{for all $x,y,z,u,w\in A$.}\end{equation}
Our proof heavily depends on this identity.

Since $A$ is semiprimitive, and hence semiprime, we may use \cite[Lemma 1.1.9]{H2} which states that the following implication holds for every $a\in A$:
\begin{equation}\label{hd}
[[a,x],a]=0\,\,\,\mbox{for all $x\in A$}\implies a\in Z(A).\end{equation}
Therefore, it follows from \eqref{e} that 
 $$[u, f([x,y]) - [x, f(y)]]\in Z(A)\quad\mbox{ for all $x,y,u\in L$.}$$ 
Hence, using \eqref{hd} once again we arrive at
$$f([x,y]) - [x, f(y)] \in Z(A)\quad\mbox{ for all $x,y\in L$.}$$ 
Thus, if  $x,y\in A$ are such that $[x,y]=0$, then $[x,f(y)]\in Z(A)$. However, a commutator lying in the center of a  semiprimitive Banach algebra must be $0$; indeed,
this follows from the Kleinecke-Shirokov (which states that if Banach algebra elements $a,b$ satisfy $[a,[a,b]]=0$, then  $[a,b]$ is quasi-nilpotent) along with the fact that
the spectral radius is submultiplicative on
commuting elements.
 Accordingly, $$(x,y)\mapsto f([x,y]) - [x, f(y)]$$ is a continuous bilinear map from $A\times A$ to $Z(A)$
for which the assumption that $A$ is zLpd  can be used.
Therefore, there exists a continuous linear map $\mu: [A,A]\to Z(A)$ such that
\begin{equation}\label{f} f([x,y]) - [x, f(y)] = \mu([x,y])\quad\mbox{for all $x,y\in A$.}\end{equation}
Extend $\mu$ to a linear map from $A$ to $Z(A)$ (which we also denote by $\mu$), and
define $\gamma:A\to A$ by
$$\gamma= f- \mu.$$
From \eqref{f} it follows that $\gamma([x,y])= [x,\gamma(y)]$. This already shows that $\gamma\in {\rm Cent}(A^-)$ since 
$\gamma([x,y])= -\gamma([y,x])= -[y,\gamma(x)] = [\gamma(x),y]$.
\end{proof}

Two questions immediately present itself. The first one is whether $\gamma$ and $\mu$ can be chosen to be continuous, and the second one is whether $\gamma$ can be described more explicitly. However, we will not go into this here. We remark that the $C^*$-algebra case is resolved in \cite[Section 6.2]{AM}.

\section{Commutativity preserving derivations}\label{s4}

This section is devoted to showing that, under relatively mild assumptions, nonzero derivations do not preserve commutativity, i.e., 
they do not map all commuting pairs of elements into commuting pairs. Our main result reads as follows.

\begin{theorem}\label{td}
Let $d$ be a nonzero derivation of a Banach algebra $A$.  Suppose that $A$ is zLpd, semiprimitive, and that there exists a unital ring $Q$ such that $A$ is a $3$-free subset of $Q$. Then there exist $x,y\in A$ such that $[x,y]=0$ and $[d(x),d(y)]\ne 0$.
\end{theorem}

\begin{proof} Suppose the theorem is not true. 
From
$$[d^2(x),y] +[x,d^2(y)] = d^2([x,y]) - 2[d(x),d(y)]$$
we see that then $[x,y]=0$ implies $[d^2(x),y] + [x,d^2(y)] =0$. Since $d$ is continuous \cite{JS} and $A$ is zLpd, there exist a continuous linear map $\rho:A\to A$ such that
$$[d^2(x),y] + [x,d^2(y)]=\rho([x,y])$$
for all $x,y\in A$. Using $[xy,z]+ [zx,y] + [yz,x]=0$ it follows that 
\begin{align*}&[d^2(xy),z] + [d^2(zx),y] + [d^2(yz),x]\\
 +& [xy,d^2(z)] + [zx,d^2(y)] + [yz,d^2(x)]=0\end{align*}
for all $x,y,z\in A$.  That is,
\begin{equation} \label{eee}E(x,y)z + E(z,x)y + E(y,z)x + zF(x,y) + yF(z,x) + xF(y,z)=0,\end{equation}
where 
\begin{equation} \label{fff}E(x,y)= d^2(xy) - d^2(x)y\quad\mbox{and}\quad F(x,y)=xd^2(y) - d^2(xy).\end{equation}
As $A$ is a $3$-free subset of $Q$, 
it follows in particular that there exist maps $p,q:A\to Q$ and $\lambda:A\times A \to Z(Q)$ such that
$$E(x,y)=xp(y)  + yq(x) + \lambda(x,y)$$
for all $x,y\in A$. Using this in \eqref{eee} we obtain 
$$xG(y,z) + yG(z,x) + zG(x,y)=0$$
for all $x,y,z\in A$, where
$$G(x,y)= F(x,y)+p(x)y + q(y)x + \lambda(x,y).$$
Again using the assumption that  $A$ is a $3$-free subset of $Q$ we get $G=0$, i.e.,
$$ F(x,y)=-p(x)y - q(y)x - \lambda(x,y)
$$ 
for all $x,y\in A$.

Applying the derivation law we see from \eqref{fff} that 
$$E(x,y)= 2d(x)d(y) + xd^2(y)\quad\mbox{and}\quad F(x,y)=-2d(x)d(y)-d^2(x)y ,
$$
and hence
\begin{equation}\label{d1}
2d(x)d(y) + xd^2(y) = xp(y)  + yq(x) + \lambda(x,y)\end{equation}
and
\begin{equation}\label{d2}
2d(x)d(y)+d^2(x)y= p(x)y +q(y)x +\lambda(x,y)
\end{equation}
for all $x,y\in A$. 
Replacing $y$ by $yz$ in \eqref{d2} we obtain
$$2d(x)d(y)z + 2d(x)yd(z) + d^2(x)yz = p(x)yz + q(yz)x + \lambda(x,yz),$$
and so, using \eqref{d2} again, we obtain
\begin{equation}\label{d2a}
2d(x)yd(z)+q(y)xz+\lambda(x,y)z= q(yz)x +\lambda(x,yz).
\end{equation}
Similarly, replacing $x$ by $zx$ in \eqref{d1} we obtain
\begin{equation*}
2d(z)xd(y)+zyq(x)+\lambda(x,y)z= yq(zx) +\lambda(zx,y).
\end{equation*}
By changing the variable notation, we can write this as 
\begin{equation}
\label{d1a}
2d(x)yd(z)+xzq(y)+\lambda(y,z)x= zq(xy) +\lambda(xy,z).
\end{equation}
Comparing \eqref{d2a} and \eqref{d1a} we arrive at
\begin{align*}&\big(q(yz) + \lambda(y,z)\big)x - \big(q(y)x\big)z  + x\big(zq(y)\big)- z\big(q(xy) + \lambda(x,y)\big) \\
=&\lambda(xy,z)- \lambda(x,yz)\in Z(Q).\end{align*}
For any fixed $y$, we can intepret this as a functional identity for which the $3$-freeness assumption is applicable.  Hence, in particular, for any $y$ there exists an $r_y\in Q$ such that
$$zq(y) - r_yz \in Z(Q)$$ for all $z\in Z(Q)$. Now, since $A$ is also $2$-free in $Q$ it follows that 
$q(y)=r_y\in Z(Q)$
for any $y\in A$. From \eqref{d1a} we now see that $d(x)yd(x)$ commutes with $x$ for all $x,y\in A$, i.e.,
$$d(x)yd(x)x=xd(x)yd(x).$$
Consequently, for all $x,y,z\in A$,
\begin{align*}
d(x)y\big((xd(x)zd(x)\big) &= d(x)\big(yd(x)z\big)d(x)x
\\ =& \big(xd(x)yd(x)\big)zd(x) = d(x)yd(x)xzd(x).
\end{align*}
That is,
$$d(x)A[d(x),x] A d(x) =\{0\}$$
and hence also 
$$[d(x),x] A [d(x),x] A [d(x),x] =0$$
for all $x\in A$. This means that, for any $x\in A$, $[d(x),x]$ generates a nilpotent ideal in $A$. Since $A$ is semiprimitive, the only possibility is that $$[d(x),x]=0$$
for all $x\in A$. Now, since $d$ preserves any primitive ideal $P$ of $A$ \cite{S}, $d$ induces a derivation $d_P$ of the Banach algebra $A/P$ given by $d_P(x + P) =d(x)+P$. Of course, $d_P$ also satisfies $[d_P(u),u]=0$ for all $u\in A/P$. Since $A/P$ is a primitive and hence a prime algebra, we may use  the well-known  Posner's theorem \cite{P}
stating that  a derivation $\delta$ of a prime ring $R$ must be $0$ if it satisfies $[\delta(u),u]=0$ for all $u\in R$. Thus,
 $d_P=0$. Since $P$ is an arbitrary primitive ideal and $A$ is semiprimitive, this leads to a contradiction that $d=0$.
\end{proof}

As pointed out in the introduction, $C^*$-algebras are zLpd. Of course, they are also semiprimitive. 
With reference to Section \ref{s2}, we now state two corollaries of Theorem \ref{td}. 

\begin{corollary}\label{td1} Let $A$ be a prime $C^*$-algebra. 
The following conditions are equivalent:
\begin{enumerate}
\item[{\rm (i)}] There exists a nonzero derivation $d$ of $A$ such that $d(x)$ and $d(y)$ commute whenever $x$ and $y$ commute.
\item[{\rm (ii)}] $\dim A\ne 1$  and every derivation $d$ of $A$ has the property  that $d(x)$ and $d(y)$ commute whenever $x$ and $y$ commute.
\item[{\rm (iii)}]
  $A\cong M_2(\mathbb C)$.
\end{enumerate}
\end{corollary}

\begin{proof}
(i)$\implies$(iii). Theorem \ref{td} shows that (i) is fulfilled only when $A\cong M_k(\mathbb C)$ with $k < 3$. Since there are no nonzero derivations on $A$ if  $k=1$,
the only possibility is that $k=2$. 

(iii)$\implies$(ii). Note that two matrices $x,y\in M_2(\mathbb C)$, with $x$ nonscalar, commute only when $y$ is a linear combination of $x$ and the identity matrix. Since
every derivation $d$ of $M_2(\mathbb C)$ is inner, this readily implies that $d(x)$ and $d(y)$ commute (in fact, they are linearly dependent) whenever $x$ and $y$ commute.

(ii)$\implies$(i). Since $A$ is prime and of dimension greater than $1$, it is not commutative. Every noncentral element obviously gives rise to a nonzero inner derivation which, by assumption, preserves commutativity.
\end{proof}

Since  von Neumann algebras have only inner derivations, we state our next corollary as follows.

\begin{corollary}\label{td2} 
Let $A$ be a von Neumann algebra with no central summands of type  $I_2$. If
$a\in A$ is such that $[a,x]$ and $[a,y]$ commute whenever $x$ and $y$ commute, then 
$a\in Z(A)$.
\end{corollary}

\section{Commutativity preserving bijective linear maps} \label{s5}

Obvious examples of bijective commutativity preserving linear maps are isomorphisms and antiisomorphisms.
More generally, so are direct sums of isomorphisms and antiisomorphisms---we say that a bijective linear map $\sigma:A\to B$ is the {\em direct sum of an isomorphism and an antiisomorphism} if $A$ is the direct sum of two ideals $I$  and $J$, $B$ is the direct sum of two ideals $I'$ and $J'$, the restriction of $\sigma$ to $I$ is an isomorphism from
$I$ onto $I'$, and the restriction of $\sigma$ to $J$ is an antiisomorphism from
$J$ onto $J'$.
Further, note that if $\sigma:A\to B$ is any commutativity preserving linear map, $\lambda\in Z(B)$  and $\tau:A\to Z(B)$ is any linear map, then the map \begin{equation}\label{thr}\theta(x)= \lambda \sigma(x)+\tau(x)\end{equation} also preserves commutativity. One usually tries to prove that a bijective linear commutativity preserving map is of the form \eqref{thr}, with $\sigma$ being the direct sum of an isomorphism and an antiisomorphism. Note that such a map preserves commutativity in both directions (i.e.,
both $\theta$ and $\theta^{-1}$ preserve commutativity), provided that $\lambda$ is not a zero divisor.
We now give an example of a bijective linear map that preserves commutativity in one direction only.

\begin{example}
Let $A$ be an algebra and let $V$ be a subspace of $A$ such that $A=V\oplus Z(A)$. Let $\gamma:Z(A)\to Z(A)$ be an injective linear map such that its range $Z_0$ is a proper subspace of $Z(A)$. Pick a subspace $Z_1$ of $Z(A)$ such that $Z(A)=Z_0\oplus Z_1$, and let $\delta$ be a bijective linear map from some  algebra $J$ onto $Z_1$. Define $\theta:A\times J\to A$ by $$\theta(v + z,u) = v + \gamma(z) + \delta(u)$$ for all $v\in V$,  $z\in Z(A)$, $u\in J$. It is easy to check  that $\theta$ is a bijective linear map that preserves commutativity. If $J$ is not commutative,  then $\theta^{-1}$
does not preserve commutativity, and hence $\theta$ is not
 of the form \eqref{thr}.
\end{example}

This example justifies the conclusion of  Theorem \ref{t}\,(a), while the next, simple and well-known  example justifies the $3$-freeness assumption. 

\begin{example}
It is easy to check that every linear map from $M_2(\mathbb C)$ to another algebra $B$ that sends the identity matrix to a central element in $B$ preserves commutativity. Of course, such a map  may not be of the standard form \eqref{thr}.
\end{example}

It is well-known that functional identities are applicable to the study of commutativity preservers. So far, the  results obtained by this approach used at least one of the following two assumptions:  
the center of the target algebra is a field (see \cite[Section 7.1]{FIbook}) or the commutativity is preserved in both directions (like in \cite{BM}). 
In the next theorem, we will be able to avoid these assumptions, however by adding the assumption that the first Banach algebra is zLpd. It should be mentioned that 
parts of the proof will be almost identical to parts of the proofs of \cite[Theorem 6.1]{FIbook} and
\cite[Theorem 3]{BM}, but there will  also be  entirely different parts.

\begin{theorem}\label{t}
Let $A$ and $B$ be unital Banach algebras. Assume that $A$ is semiprime and zLpd, and that both $A$ and $B$ are $3$-free.
If $\theta:A\to B$ is a bijective  continuous commutativity preserving linear map (i.e., $[x,y]=0$ implies $[\theta(x),\theta(y)]=0$), then one of the following statements holds:
\begin{enumerate}
\item[{\rm (a)}] 
$A$ contains a noncentral ideal $J$ such that $\theta(J) \subseteq Z(B)$ (so $\theta$ preserves commutativity in one direction only).
 \item[{\rm (b)}] 
$\theta$ is of the form \eqref{thr} with $\lambda$  an invertible element in $Z(B)$, $\tau$  a linear map from $A$ to $Z(B)$, and 
$\sigma:A\to B$ the direct sum  of an isomorphism and an antiisomorphism (so  $\theta$ preserves commutativity in both directions). Moreover, $\sigma$ and $\tau$ are continuous.
\end{enumerate}
\end{theorem}


\begin{proof}
Since $A$ is zLpd, there exists a  continuous linear map $\rho:A\to B$ satisfying $$[\theta(x),\theta(y)]=\rho([x,y])$$ for all $x,y\in A$. 
Using $[xy,z] + [zx,y] + [yz,x]=0$ it follows that
\begin{equation}\label{e1a}[\theta(xy),\theta(z)] +  [\theta(zx),\theta(y)]+ [\theta(yz),\theta(x)] =0\end{equation}
for all $x,y,z\in A$. Since $B$ is $3$-free,
we may use \cite[Theorem 4.13]{FIbook} to conclude that there exist elements $\lambda_1,\lambda_2\in Z(B)$ and maps $\mu_1,\mu_2:A\to Z(B)$, $\nu:A^2\to Z(B)$ such that
\begin{equation}\label{e1b}\theta(xy) = \lambda_1 \theta(x) \theta(y) + \lambda_2 \theta(y) \theta(x) + \mu_1(x)\theta(y) + \mu_2(y)\theta(x) + \nu(x,y)\end{equation}
for all $x,y\in A$. Moreover,  \cite[Lemma 4.6 and Remark 4.7]{FIbook}
show that $\mu_1$ and $\mu_2$ are linear are $\nu$ is bilinear. Using \eqref{e1b} in \eqref{e1a} we obtain
$$(\mu_1-\mu_2)(x)[\theta(y),\theta(z)] + (\mu_1-\mu_2)(y)[\theta(z),\theta(x)]+(\mu_1-\mu_2)(z)[\theta(x),\theta(y)]=0$$
for all $x,y,z\in A$, which implies that $\mu_1=\mu_2$ by \cite[Lemma 4.4]{FIbook}. 
Thus,
\begin{equation}\label{e1}\theta(xy) = \lambda_1 \theta(x) \theta(y) + \lambda_2 \theta(y) \theta(x) + \mu(x)\theta(y) + \mu(y)\theta(x) + \nu(x,y)\end{equation}
for all $x,y\in A$, 
where $\mu=\mu_1=\mu_2$. This yields
\begin{equation}\label{llie}\theta([x,y]) = (\lambda_1-\lambda_2)[\theta(x),\theta(y)] + \nu(x,y)-\nu(y,x).\end{equation}

 Let $(x_n)$ be a sequence in $A$ such that $\lim x_n =0$ and $\lim \mu(x_n) = c$ for some $c\in Z(B)$. Since $\theta$ is continuous and 
 $\nu$ maps 
into $Z(B)$, it follows from \eqref{e1} that $c[\theta(y),w]=0$ for all $y\in A$, $w\in B$. The  $2$-freeness of $B$ yields  $c=0$, so $\mu$ is continuous by the closed graph theorem. 

In the next step, we use \eqref{e1} to compute $\theta(xyz)$ in two ways. First we have
\begin{align*}
\theta((xy)z) = &\lambda_1\theta(xy)\theta(z)+\lambda_2\theta(z)\theta(xy) + \mu(xy)\theta(z) + \mu(z)\theta(xy) + \nu(xy,z)\\
=& \lambda_1^2 \theta(x)\theta(y)\theta(z)+ \lambda_1\lambda_2 \theta(y)\theta(x)\theta(z)+ \lambda_1\mu(x)\theta(y)\theta(z) \\
&+ \lambda_1\mu(y)\theta(x)\theta(z)+\lambda_1\nu(x,y)\theta(z)+
\lambda_1\lambda_2 \theta(z)\theta(x)\theta(y)\\
&+ \lambda_2^2 \theta(z)\theta(y)\theta(x)
+\lambda_2\mu(x)\theta(z)\theta(y) + \lambda_2\mu(y)\theta(z)\theta(x) \\
&+ \lambda_2\nu(x,y)\theta(z)+\mu(xy)\theta(z) + \lambda_1\mu(z)\theta(x)\theta(y) \\&+ \lambda_2\mu(z)\theta(y)\theta(x)+ \mu(x)\mu(z)\theta(y) 
+ \mu(y)\mu(z)\theta(x)\\& + \mu(z)\nu(x,y) +\nu(xy,z). 
\end{align*}
On the other hand,
\begin{align*}
\theta(x(yz)) = &\lambda_1\theta(x)\theta(yz)+\lambda_2\theta(yz)\theta(x) + \mu(x)\theta(yz) + \mu(yz)\theta(x) + \nu(x,yz)\\
=& \lambda_1^2 \theta(x)\theta(y)\theta(z)+ \lambda_1\lambda_2 \theta(x)\theta(z)\theta(y)+ \lambda_1\mu(y)\theta(x)\theta(z) \\
&+ \lambda_1\mu(z)\theta(x)\theta(y)+\lambda_1\nu(y,z)\theta(x)+
\lambda_1\lambda_2 \theta(y)\theta(z)\theta(x)\\
&+ \lambda_2^2 \theta(z)\theta(y)\theta(x)
+\lambda_2\mu(y)\theta(z)\theta(x) + \lambda_2\mu(z)\theta(y)\theta(x) \\
&+ \lambda_2\nu(y,z)\theta(x) + \lambda_1\mu(x)\theta(y)\theta(z) + \lambda_2\mu(x)\theta(z)\theta(y) \\& + \mu(x)\mu(y)\theta(z)+ \mu(x)\mu(z)\theta(y) + \mu(x)\nu(y,z)\\& +\mu(yz)\theta(x)+\nu(x,yz). 
\end{align*}
Comparing the two expressions we arrive at
\begin{align}&\lambda_1\lambda_2[\theta(y),[\theta(x),\theta(z)]]\nonumber\\
 +& \big((\lambda_1+\lambda_2)\nu(x,y) + \mu(xy)-\mu(x)\mu(y)\big)\theta(z) \label{lla} \\
-& \big((\lambda_1+\lambda_2)\nu(y,z) + \mu(yz)-\mu(y)\mu(z)\big)\theta(x) \in Z(B)\nonumber\end{align}
for all $x,y,z\in A$. Let us show that this implies that
\begin{equation}\label{ee'}\lambda_1\lambda_2=0\end{equation}
and 
 \begin{equation}\label{e2}
(\lambda_1 + \lambda_2)\nu(x,y)  + \mu(xy)=\mu(x)\mu(y)
\end{equation}
for all $x,y\in A$. It is enough to prove \eqref{ee'} since \eqref{e2} then follows  (for any fixed $y$) from \cite[Lemma 4.4]{FIbook}. Suppose \eqref{ee'} is not true,
i.e., $\lambda_1\lambda_2\ne 0$. Since $B$ is $2$-free, then there exists a $y_0\in A$ such that
$$b=\lambda_1\lambda_2\theta(y_0)\notin Z(B).$$ Writing $y_0$ for $y$ in \eqref{lla} we get
$$E_1(x)\theta(z) + E_2(z)\theta(x) + \theta(z)F_1(x) + \theta(x)F_2(z)\in Z(B)$$
for all $x,z\in A$, where
\begin{align*}E_1(x)&=b\theta(x),\\E_2(z)&=-b\theta(z) - (\lambda_1+\lambda_2)\nu(y_0,z) - \mu(y_0z) + \mu(y_0)\mu(z),\\
F_1(x)&=\theta(x)b + (\lambda_1+\lambda_2)\nu(x,y_0) + \mu(xy_0) - \mu(x)\mu(y_0),\\
F_2(z)&=-\theta(z)b.
\end{align*}
Since $B$ is $3$-free it follows from \cite[Theorem 4.3]{FIbook} that, in particular, $E_1$ is of the form $E_1(x)=\theta(x)d + \gamma(x)$
for some $d\in B$
and $\gamma:A\to Z(B)$. Consequently, $$b\theta(x) - \theta(x)d =\gamma(x)\in Z(B)$$ for all $x\in A$. However, the $2$-freeness of $B$  now implies that $b\in Z(B)$, contrary to our choice of $y_0$. We have thereby proved that \eqref{ee'} and \eqref{e2} indeed hold.

Let $L = \theta^{-1}(Z(B))$. From \eqref{e1} we see that $L$ is closed under multiplication, so it is a subalgebra of $A$. 
Moreover, from \eqref{llie} we see that $$\theta([x,\ell]) = \nu(x,\ell)-\nu(\ell,x) \in Z(B)$$ for all $x\in A$ and $\ell\in L$, meaning that $[A,\ell]\subseteq L$. 
That is, $L$ is a Lie ideal of $A$.
 We are now in a position to use the classical Herstein's theory of Lie ideals. Specifically, since $A$ is  semiprime,  \cite[Lemma 1.3]{Her} shows that either $L$ is contained in $Z(A)$ or
it  contains a noncentral ideal $J$. Thus, either   (a) holds or $L\subseteq Z(A)$.
As $\theta$ preserves commutativity, we trivially have $Z(A)\subseteq L$, so $L=Z(A)$ in the latter case. That is,  either (a) holds or $Z(B)=\theta(Z(A))$, i.e.,
\begin{equation}\label{e3}
z\in Z(A) \iff \theta(z)\in Z(B).
\end{equation}
We may therefore assume that \eqref{e3} is true. Our goal now is to derive  (b) from it.

We claim that \begin{equation}\label{no1}[\theta^{-1}(wy),\theta^{-1}(y)]=0\end{equation} for all $y\in B$ and $w\in Z(B)$. Indeed, from \eqref{llie} it follows that
$$\theta([\theta^{-1}(wy),\theta^{-1}(y)])\in Z(B),$$
and hence $$[\theta^{-1}(wy),\theta^{-1}(y)]\in Z(A)$$ by \eqref{e3}. Using \cite[Proposition 3.1]{B2} it follows that \eqref{no1} indeed holds.
In a similar fashion, but using \cite[Theorem 2]{B1} instead of \cite[Proposition 3.1]{B2}, we see that  \begin{equation}\label{no2}[\theta^{-1}(y^2),\theta^{-1}(y)]=0\end{equation} for every $y\in B$.



Let $z\in Z(A)$. Then $[zx,x]=0$ for any $x\in A$, and so $[\theta(zx),\theta(x)] =0$. As $B$ is $2$-free,
there exists a $w\in Z(B)$ such that $\theta(zx) - w\theta(x)\in Z(B)$ for all $x\in A$ \cite[Corollary 4.15]{FIbook}. The $2$-freeness also implies that 
$w$ is uniquely determined. We may therefore define $\alpha:Z(A)\to Z(B)$ by $\alpha(z)=w$. Thus, $\alpha$ satisfies 
\begin{equation}\label{em}
\theta(zx)-\alpha(z)\theta(x)\in Z(B) 
\end{equation}
for all $x\in A$ and $z\in Z(A)$. We claim that $\alpha$ is linear.
Indeed, from $\theta((z+z')x)= \theta(zx) + \theta(z'x)$, where $z,z'\in Z(A)$ and $x\in A$, we infer that
$$\big(\alpha(z+z') - \alpha(z) - \alpha(z')\big)\theta(x)\in Z(B),$$
and hence $\alpha(z+z') = \alpha(z) + \alpha(z')$ since $B$ is $2$-free. Similarly we see that $\alpha$ is homogeneous. 
If $\alpha(z)=0$, then $\theta(zx)\in Z(B)$ for every $x\in A$, so $zx\in Z(A)$ by \eqref{e3}, and hence $z=0$ as $A$ is $2$-free. Finally, we claim that $\alpha$ is surjective.
To prove this, note that $\theta^{-1}$ satisfies an analogous condition \eqref{no1} and so, since $A$ is $2$-free, there exists a map $\beta:Z(B)\to Z(A)$ satisfying
\begin{equation}\label{bn}\theta^{-1}(wy) - \beta(w)\theta^{-1}(y)\in Z(A)\end{equation}
for all $w\in Z(B)$, $y\in B$. Applying $\theta$ to this relation we obtain
$$wy - \theta(\beta(w)\theta^{-1}(y))\in Z(B),$$ which in turn implies that
$$\big(w - \alpha(\beta(w))\big) y \in Z(B).$$
As $B$ is $2$-free,  this gives $w=\alpha(\beta(w))$ for every $w\in Z(B)$. Thus, $\alpha$ is surjective, and so is a linear isomorphism 
from $Z(A)$ onto $Z(B)$ (incidentally, it can be easily shown that it is even an algebra isomorphism, but we will not need this). 

Next, since, by \eqref{no2},  $\theta^{-1}(y^2)$ and $\theta^{-1}(y)$  commute for every $y\in B$ and $A$ is $3$-free, \cite[Corollary 4.15]{FIbook} implies that there
exist an $\omega\in Z(A)$ and a linear map $\eta:B\to Z(A)$ 
 such that
$$\theta^{-1}(y^2) - \omega \theta^{-1}(y)^2 - \eta(y)\theta^{-1}(y) \in Z(A)$$
for all $y\in B$. Applying $\theta$ to this relation and using \eqref{em} we obtain
$$y^2 - \lambda\theta(\theta^{-1}(y)^2) - \alpha(\eta(y))y \in Z(B)$$
where $$\lambda = \alpha(\omega)\in Z(B).$$
On the other hand,  \eqref{e1} shows that 
$$\theta(\theta^{-1}(y)^2) - (\lambda_1+\lambda_2) y^2 - 2 \alpha(\mu(\theta^{-1}(y)))y \in Z(B).$$
Comparing these two relations involving $\theta(\theta^{-1}(y)^2)$ it follows that
$$(1-\lambda(\lambda_1+\lambda_2))y^2 - \big(\alpha(\eta(y)) + 2\alpha(\mu(\theta^{-1}(y))\big)y \in Z(B)$$
for all $y\in B$. Since $B$ is $3$-free, linearizing this identity and then using \cite[Lemma 4.4]{FIbook} we obtain $1-\lambda(\lambda_1+\lambda_2)=0$. Thus, $\lambda$ is invertible and
$$\lambda^{-1}=\lambda_1+\lambda_2.$$

Define $\varphi:A\to B$ by
$$\varphi(x) = \lambda_1\big(\theta(x) + \lambda \mu(x)\big).$$
Using \eqref{e1} along with \eqref{ee'} we obtain
$$\varphi(xy) = \lambda_1^2 \theta(x)\theta(y) + \lambda_1\mu(x)\theta(y) 
+\lambda_1\mu(y)\theta(x) + \lambda_1\nu(x,y)+
\lambda_1\lambda \mu(xy)$$
and
$$\varphi(x)\varphi(y) = \lambda_1^2 \theta(x)\theta(y) + \lambda_1^2 \lambda \mu(x)\theta(y) +  \lambda_1^2 \lambda \mu(y)\theta(x) + \lambda_1^2\lambda^{2}
\mu(x)\mu(y).
$$
Since $\lambda = (\lambda_1+\lambda_2)^{-1}$ and $\lambda_1\lambda_2=0$, $$\lambda_1^2 \lambda = \lambda_1.$$ Together with \eqref{e2}, this implies that
$\varphi(xy) = \varphi(x)\varphi(y)$. Thus, $\varphi$ is an algebra homomorphism. Since $\theta$ and $\mu$ are continuous, so is $\varphi$.

Similarly, we see that  $\psi:A\to B$ given by
$$\psi(x) = \lambda_2\big(\theta(x) + \lambda \mu(x)\big)$$
is a continuous algebra antihomomorphism.

We now define
 $$\sigma=\varphi  + \psi.$$ 
 Note that $$\sigma(x)= (\lambda_1+\lambda_2)\theta(x) + \mu(x),$$ and hence
\begin{equation}\label{jk} \theta(x) = \lambda\sigma(x) + \tau(x)\end{equation} where $$\tau(x)=-\lambda\mu(x)\in Z(B);$$
obviously, $\tau$ is linear and continuous. 
From  \eqref{e3} we easily infer that for any $z\in A$,
\begin{equation}\label{nn} z\in Z(A)\iff \sigma(z)\in Z(B).\end{equation}
Observe also that $\lambda_1\lambda_2=0$ shows that $\varphi(x)\psi(y)=\psi(y)\varphi(x)=0$ for all $x,y\in A$, which in turn implies that
\begin{equation}\label{za}
\sigma(zx)=\sigma(z)\sigma(x)\,\,\,\,\mbox{and}\,\,\,\, \sigma(x^2) = \sigma(x)^2
\end{equation}
for all $z\in Z(A)$, $x\in A$. 

Our next goal is to show  that $\sigma$ is bijective. Suppose $\sigma(c)=0$ for some $c\in A$. Then $c\in Z(A)$ 
by \eqref{nn}, and so
\eqref{za} shows that $\sigma(cx) =0$ for every $x\in A$. Using \eqref{nn} again, we have $cA\subseteq Z(A)$. However, this implies $c=0$ since $A$ is $2$-free. It remains to prove that
$\sigma$ is surjective. To this end, we first show that the restriction of $\sigma$ to $Z(A)$ is equal to $\alpha$. From \eqref{jk} and \eqref{za} it follows that
\begin{align*}
\theta(zx) =& \lambda\sigma(zx) +\tau(zx) = \lambda\sigma(z)\sigma(x) + \tau(zx) \\=& \sigma(z)\theta(x) - \sigma(z)\tau(x) + \tau(zx)\end{align*}
for all $z\in Z(A)$, $x\in A$. Thus, $\theta(zx) - \sigma(z)\theta(x) \in Z(B)$, which together with \eqref{em} shows that
$(\sigma(z) - \alpha(z))\theta(x)\in Z(B)$. But then $\sigma(z)=\alpha(z)$ since  $A$ is $2$-free. Now, we know that $\alpha$ is a bijection from $Z(A)$ onto $Z(B)$, so we can write $\lambda = \sigma(\alpha^{-1}(\lambda))$ and $\tau(x) =  \sigma(\alpha^{-1}(\tau(x))$. Using \eqref{jk} and \eqref{za}, we thus have
$$\theta(x)= \sigma(\alpha^{-1}(\lambda))\sigma(x) + \sigma(\alpha^{-1}(\tau(x)) = \sigma\big(\alpha^{-1}(\lambda)x + \alpha^{-1}(\tau(x))\big).$$
Since $\theta$ is surjective, it follows that so is $\sigma$.

Observe that $f=\lambda_1\lambda$ is a central idempotent, $1-f=\lambda_2\lambda$, $f\lambda_1  = \lambda_1$ and hence $f\varphi(x)=\varphi(x)$ for every $x\in A$;
similarly, $(1-f)\lambda_2  = \lambda_2$ and hence $(1-f)\psi(x)=\psi(x)$ for every $x\in A$. From \eqref{nn} and \eqref{za} we see that $e=\sigma^{-1}(f)$ is also a central idempotent. Since $\sigma(1)=1$ by \eqref{za}, we have $1-e = \sigma^{-1}(1-f)$. Next, using \eqref{za} we see that $$\sigma(ex)= f\sigma(x)= f\sigma(ex) = \varphi(ex).$$ Thus, the restriction of $\sigma$ to $eA$ is equal to $\varphi$, and $\sigma$ maps $I=eA$ onto $I'=fB$. Similarly,
the restriction of $\sigma$ to $(1-e)A$ is equal to $\psi$, and $\sigma$ maps $J=(1-e)A$ onto $J'=(1-f)B$. Thus, $\sigma$ is the direct sum of an isomorphism and an antiisomorphism.
\end{proof}

Our main examples of  Banach algebras satisfying the conditions of Theorem \ref{t} are von Neumann algebras with no central summands of type $I_1$ and $I_2$. Let us state explicitly the corresponding corollary. It should be mentioned that this result is known in the case where $\theta$ preserves commutativity in both directions
 \cite{BM} (see also \cite[Section 6.5]{AM}).

\begin{corollary}\label{tc}
Let $A$ and $B$ be  von Neumann algebras with no central summands of type $I_1$ and $I_2$. If $\theta:A\to B$ is a bijective  continuous commutativity preserving linear map, then one of the following statements holds:
\begin{enumerate}
\item[{\rm (a)}] 
$A$ contains a noncentral ideal $J$ such that $\theta(J) \subseteq Z(B)$ (so $\theta$ preserves commutativity in one direction only).
 \item[{\rm (b)}] 
$\theta$ is of the form \eqref{thr} with $\lambda$  an invertible element in $Z(B)$, $\tau$  a linear map from $A$ to $Z(B)$, and 
$\sigma:A\to B$  the direct sum  of an isomorphism and an antiisomorphism (so  $\theta$ preserves commutativity in both directions). Moreover, $\sigma$ and $\tau$ are continuous.
\end{enumerate}
\end{corollary}

We will not consider corollaries concerning prime  algebras since rather definitive results about them are already known \cite[Section 7.1]{FIbook}. Let us conclude by mentioning
that if $R$ is a unital cyclically amenable Banach algebra, then $A=M_n(R)$, $n\ge 3$, is both zLpd  \cite[Theorem 4.4]{ABEV} and $3$-free; moreover, if $R$ is semiprime, then so is $A$. Thus, there are other classes of Banach algebras to which Theorem \ref{t} is applicable.




\bigskip
\noindent
{\bf Acknowledgment}. The author would like to thank the referee for useful suggestions.

\end{document}